\documentclass[a4paper,11pt]{amsart}

\usepackage[english]{babel}

\usepackage[utf8]{inputenc}  
\usepackage[T1]{fontenc}
\usepackage{pifont}

\usepackage{palatino}

\usepackage{amsmath}
\usepackage{amssymb}
\usepackage{amsthm}
\usepackage{amscd}

\usepackage[abbrev,backrefs]{amsrefs}

\usepackage{mathtools}

\usepackage{tikz}

\newtheorem{theoreme}                        {Theorem} 
\newtheorem{proposition}{Proposition} [section]

\newtheorem{lemme}          {Lemma} [section]
\newtheorem{definition} {Definition} [section]
\theoremstyle{definition}
\newtheorem{exemple}    {Example} [section]

\theoremstyle{remark}

\renewcommand{\d}{\mathrm d}

\newcommand{\R}{\mathbb{R}} 
\newcommand{\N}{\mathbb{N}}

\newcommand{\Sp}{\mathbb{S}}

\newcommand{\Ha}{\mathcal{H}} 

\newcommand{\Newton}{\mathcal{N}}
\newcommand{\restr}[1]{{\lfloor_{_{#1}}}} 
\newcommand{\imset}[2]{\left\{{#1}:{#2}\right\}} 
\newcommand{\compset}[2]{\left\{{#1}~\vert~{#2}\right\}} 
\newcommand{\e}{\mathrm{e}}

\newcommand{\norm}[1]{\left\vert{#1}\right\vert}

\DeclareMathOperator{\diam}{diam}

\numberwithin{equation}{section}

\begin{document}

\title{A decomposition for Borel measures \(\mu \le \mathcal{H}^{s}\)}

\author{Antoine Detaille}
\address{
Antoine Detaille\hfill\break\indent
Université catholique de Louvain\hfill\break\indent
Institut de Recherche en Mathématique et Physique\hfill\break\indent
Chemin du cyclotron 2, L7.01.02\hfill\break\indent
1348 Louvain-la-Neuve, Belgium}
\email{Antoine.Detaille@student.uclouvain.be}

\author{Augusto C. Ponce}
\address{
Augusto C. Ponce\hfill\break\indent
Université catholique de Louvain\hfill\break\indent
Institut de Recherche en Mathématique et Physique\hfill\break\indent
Chemin du cyclotron 2, L7.01.02\hfill\break\indent
1348 Louvain-la-Neuve, Belgium}
\email{Augusto.Ponce@uclouvain.be}

\begin{abstract}
We prove that every finite Borel measure \( \mu \) in \( \R^N \) that is bounded from above by the Hausdorff measure \( \Ha^s \) can be split in countable many parts \( \mu\restr{E_k} \) that are bounded from above by the Hausdorff content \( \Ha^s_\infty \).
Such a result generalises a theorem due to R.~Delaware that says that any Borel set with finite Hausdorff measure can be decomposed as a countable disjoint union of straight sets.	
We apply this decomposition to show the existence of solutions of a Dirichlet problem involving an exponential nonlinearity.
\end{abstract}

\subjclass[2020]{Primary: 28A78; Secondary: 28A12}

\keywords{Hausdorff measure, Hausdorff content, straight set}

\maketitle{}


\section{Introduction and main result}

Let $0 \leq s < +\infty$ and \(N \in \N_\ast\) with \(\N_{\ast} \vcentcolon= \N\setminus\{0\}\). 
Our goal in this paper is to exploit the concepts of Hausdorff measure \(\Ha^{s}\) and content \(\Ha_{\infty}^{s}\) to obtain some better understanding about the finite Borel measures \(\mu\) in \(\R^{N}\) that satisfy the estimate 
\begin{equation}
\label{eqHaudorffEstimate}
\mu(E){}
\le \Ha^{s}(E)
\quad \text{for every Borel set \(E \subset \R^{N}\).}
\end{equation}
The definitions of \(\Ha^{s}\) and \(\Ha_{\infty}^{s}\) are recalled in Section~\ref{section2} and we implicitly assume throughout the paper that our measures are nonnegative.
Observe that the condition \(\mu \le \Ha^{s}\) above is not only satisfied by measures \(\mu\) of the form \(\mu = f \Ha^{s}\) with a measurable function \(f \le 1\), but also by any \(\mu\) such that \(\mu(E) = 0\) for every Borel set \(E \subset \R^{N}\) with finite \(\Ha^{s}\) measure.
In particular, one may have \(\mu = g \Ha^{t}\) with \(t > s\), regardless of the coefficient \(g\).{}

We want to know more precisely in what extent \(\Ha^{s}\) may be replaced in \eqref{eqHaudorffEstimate} by the smaller quantity that is \(\Ha_{\infty}^{s}\) and whether there is a unifying principle that would in particular cover the examples we mentioned above.{}
While one has \(\Ha_{\infty}^{s}(E) = 0\) if and only if \(\Ha^{s}(E) = 0\), it turns out that   \(\Ha_{\infty}^{s}\) and \(\Ha^{s}\) are quite different in general.
For example, \(\Ha_{\infty}^{s}\) is finite on every bounded subset of \(\R^{N}\), which is far from being the case for \(\Ha^{s}\) when \(s < N\).

They do coincide in some special cases of positive \(\Ha^{s}\) measure that are inherited from the case \(s = N\).
Indeed, for every \(\ell \in \{0, \dots, N\}\), if \( T \subset \R^N \) is an \(\ell\)-dimensional affine hyperplane and \( B_r(a) \subset \R^N \) is the open ball centered at \(a \in T \) with radius \( r > 0 \), then a straightforward argument gives
\[
	\Ha^\ell_\infty(B_r(a) \cap T) = \Ha^\ell(B_r(a) \cap T) = \omega_\ell r^\ell,
\]
where \(\omega_{\ell}\) denotes the volume of the \(\ell\)-dimensional unit ball.
One then deduces more generally that, for every bounded Borel set \( A \subset T \),
\begin{equation}
\label{sStraightEquation}
	\Ha^\ell_\infty(A) = \Ha^\ell(A);
\end{equation}
see Proposition~\ref{sstraightsubsets} below.
These are examples of \emph{straight sets}, a notion introduced by J.~Foran in~\cite{foran}:

\begin{definition}
\label{definitionStraightSet}
A Borel set $E \subset \R^N$ is said to be \emph{$s$-straight} whenever 
\[
	\Ha^s_\infty(E) = \Ha^s(E) < +\infty.
\]
\end{definition}

It is worth signalling that, the inequality $\Ha^s_\infty(E) \leq \Ha^s(E)$ being always satisfied, one only needs to check the reverse inequality when proving that a set is $s$-straight.
The concept of \(s\)-straight set carries an implicit idea of flatness as we have seen in~\eqref{sStraightEquation}, which is emphasised for instance by the fact that the unit sphere \( \Sp^{N-1} \) is not \((N-1)\)-straight since
\[
	\Ha^{N-1}_\infty(\Sp^{N-1}) = \omega_{N-1} \le \frac{1}{2} \Ha^{N-1}(\Sp^{N-1}).
\]

R.~Delaware established in~\cite{delaware}*{Theorem~5} the surprising property that \(s\)-straight sets are the building blocks of any Borel set with finite \(\Ha^{s}\) measure, whence settling a conjecture made by J.~Foran.
More precisely, he proved the following

\begin{theoreme}
\label{thmdelaware}
If $E \subset \R^N$ is a Borel set of finite $\Ha^s$ measure, then there exists a sequence of disjoint Borel sets $(E_n)_{n \in \N}$ such that $E = \bigcup\limits_{n \in \N} E_n$ and $E_n$ is $s$-straight for each $n \in \N$.
\end{theoreme}

A generalisation of the notion of \(s\)-straight set to Borel measures can be achieved based on the observation that every subset of an \(s\)-straight set is also \(s\)-straight by Proposition~\ref{sstraightsubsets} below.
This property implies that a Borel set $E \subset \R^N$ of finite $\Ha^s$ measure is $s$-straight if and only if
\[{}
\Ha^{s}(E \cap A) \le \Ha^s_\infty(A){}
\quad \text{for every Borel set \(A \subset \R^{N}\).}
\]
Notice that this condition still makes sense if we replace $\Ha^s$ by a Borel measure $\mu$.
Inspired by the strategy from~\cite{delaware}, we prove that Delaware's theorem has a valid counterpart in this larger framework : 
\begin{theoreme}
\label{thmdecomposition}
If $\mu$ is a finite Borel measure on $\R^N$ such that $\mu \leq \Ha^s$, then there exists a sequence of disjoint Borel sets $(E_n)_{n \in \N}$ such that $\R^N = \bigcup\limits_{n \in \N} E_n$ and, for every \(n \in \N\), 
\[{}
\mu(E_{n} \cap A)
\le \Ha^{s}_{\infty}(A)
\quad \text{for every Borel set \(A \subset \R^{N}\).}
\]
\end{theoreme}

One deduces Theorem~\ref{thmdelaware} by taking as \(\mu\) the restriction of \(\Ha^{s}\) to \(E\).{}
We would like to stress that Theorem~\ref{thmdecomposition} applies to measures that are not necessarily of Hausdorff nature, for instance they can be merely diffuse with respect to some capacity. 
In this respect, Theorem~\ref{thmdecomposition} cannot be obtained as a consequence of the decomposition given by Theorem~\ref{thmdelaware}.

Aside from the natural generalisation of Delaware's theorem, our motivation for establishing Theorem~\ref{thmdecomposition} is that it provides
one with an elegant tool to prove the existence of a solution in the sense of distributions for the  Dirichlet problem 
\begin{equation}
	\label{eqPDE-BLOP}
\left\{
\begin{alignedat}{2}
	- \Delta u + (\mathrm{e}^{u} - 1) & = \nu &&\quad \text{in \(\Omega\),}\\
	u & = 0 &&\quad \text{on \(\partial\Omega\),}
\end{alignedat}
\right.
\end{equation}
where  \(\Omega \subset \R^{N}\), \(\nu\) is a finite Borel measure in \(\Omega\) and \(\Delta = \sum\limits_{i = 1}^{N}{\partial^{2}/\partial x_{i}^{2}}\) denotes the classical Laplacian.
It has been proved by J.~L.~Vázquez~\cite{vazquez} for \(N = 2\) and by D.~Bartolucci et al.~\cite{BLOP} for \(N \ge 3\) that \eqref{eqPDE-BLOP} always admits a solution when 
\begin{equation}
	\label{eq-BLOP}
	\nu \le 4\pi \Ha^{N - 2}.{}
\end{equation}
Using Theorem~\ref{thmdecomposition} we provide a new proof of these results:

\begin{theoreme}
	\label{thmBLOP}
	Let \(N \ge 2\) and \(\Omega \subset \R^{N}\) be a smooth bounded open set.{}
	If \(\nu\) is a finite measure in \(\Omega\) that satisfies \eqref{eq-BLOP}, then there exists a function \(u\) in the Sobolev space \(W_{0}^{1, 1}(\Omega)\) such that \(\mathrm{e}^{u} \in L^{1}(\Omega)\) and
	\begin{equation}
		\label{eq-Distributions}
		- \Delta u + (\mathrm{e}^{u} - 1)  = \nu \quad \text{in the sense of distributions in \(\Omega\).}
	\end{equation}
\end{theoreme}

The original proof in \cite{BLOP} is quite involved and is based on approximation schemes of both the exponential term and the measure in the equation. 
Moreover, the manipulation of the measure itself relies on a technical decompostion lemma in that paper and the notion of reduced measure introduced in \cite{BMP}.
In contrast, the argument we present in Section~\ref{sectionBLOP} is more transparent as it relies solely on the equation we wish to solve and on an approximation of the measure that is implicity provided by Theorem~\ref{thmdecomposition}; see Lemma~\ref{lemmaDecompositionLocal} below.

A standard argument in Measure Theory extends Theorem~\ref{thmdecomposition} for \(\sigma\)-finite measures \(\mu\).
One may wonder whether it is possible to have a counterpart of Theorem~\ref{thmdecomposition} that includes the non \(\sigma\)-finite case.
To achieve this one should allow the possibility of dealing with an uncountable family of Borel sets that is suitably chosen in terms of \(\mu\).
Using Zorn's lemma, we show that

\begin{theoreme}
\label{thmdecomposition2}
If \(\mu\) is a Borel measure on \(\R^N\) such that \( \mu \leq \Ha^s \), then \(\R^N\) may be written as a disjoint union of Borel sets 
\( \R^N = \bigcup\limits_{i \in I} E_i \cup F \)
where, for each \( i \in I \), we have \( 0 < \mu(E_i) < +\infty \) and, for every Borel set \( A \subset \R^N \), 
\[
	\mu(E_i \cap A) \leq \Ha^s_\infty(A) \quad \text{and} \quad \mu(F \cap A) \in \{0,+\infty\}.
\]
\end{theoreme}

Note that if \(\mu\) is \(\sigma\)-finite, then the index set \(I\) must be countable and \(F\) is negligible with respect to \( \mu \).
We thus recover Theorem~\ref{thmdecomposition} for \(\sigma\)-finite measures.

In the next section, we explore in more detail the generalisation of \(s\)-straight sets to Borel measures that we announced above.
Then, in Section~\ref{section3} we develop some tools in preparation for the proofs of Theorems~\ref{thmdecomposition} and~\ref{thmdecomposition2} that we present in Section~\ref{section4}.

\section{From straight sets to measures}
\label{section2}

We begin by recalling the definition~\cite{hausdorff} of the \emph{Hausdorff measure} \(\Ha^s\) of dimension $s$ of a set \(E \subset \R^{N}\) as
\begin{equation}
	\label{eqHaudorffMeasure}
	\Ha^s(E) = \lim_{\delta \to 0} \Ha^s_\delta(E) = \sup_{\delta > 0}\Ha^s_\delta(E),
\end{equation}
where we take the \emph{Hausdorff capacity} $\Ha^s_\delta$ defined for $E \subset \R^N$ by
\[
	\Ha^s_\delta(E) = \inf\imset{\sum_{n \in \N} \omega_s r_n^s}{E \subset \bigcup_{n \in \N} B_{r_n}(x_n),~ 0 \leq r_n \leq \delta}.
\]
We use here the normalization constant
\(
	\omega_s = {\pi^{\frac{s}{2}}}/{\Gamma(\frac{s}{2} + 1)}
\)
that gives the volume of the $s$-dimensional unit ball when \(s\) is an integer.
Since the quantity $\Ha^s_\delta(E)$ increases as $\delta$ decreases, the limit in \eqref{eqHaudorffMeasure} always exists in \([0, +\infty]\) and also
\[{}
\Ha^s_\infty(E) \le \Ha^s_\delta(E) \le \Ha^s(E).
\]

The goal of this section is to study in more detail those measures satisfying \( \mu \leq \Ha^s \) or \( \mu \leq \Ha^s_\infty \) and to illustrate through examples some differences between both conditions.
Given a Borel measure $\mu$ on $\R^N$ and a Borel set \(E \subset \R^{N}\), we first recall that the restriction of $\mu$ to $E$ is the Borel measure $\mu\restr{E}$ defined for $A \subset \R^N$ by
\[
	\mu\restr{E}(A) = \mu(E \cap A).
\]

\begin{exemple}
\label{sStraightSetsAndMeasures}
If \( E \subset \R^N \) is a Borel set, then the Borel measure \( \mu = \Ha^s\restr{E} \) satisfies \( \mu \leq \Ha^s \).
On the other hand, we have \( \mu \leq \Ha^s_\infty \) if and only if the Borel set \( E \) is \(s\)-straight.
This relies on the fact that every subset of an \(s\)-straight set is \(s\)-straight; see \cite{foran}*{Theorem~1} :
\end{exemple}

\begin{proposition}
\label{sstraightsubsets}
If $E \subset \R^N$ is an $s$-straight Borel set, then every Borel set $A \subset E$ is also $s$-straight.
\end{proposition}
\begin{proof}
Since $\Ha^s$ is Borel regular and \( \Ha^s_\infty \) is subadditive, we have
\[
	\Ha^s(E) = \Ha^s(A) + \Ha^s(E \setminus A) \geq \Ha^s_\infty(A) + \Ha^s_\infty(E \setminus A) \geq \Ha^s_\infty(E).
\]
Using that \( E \) is \(s\)-straight, all the inequalities above are actually equalities, which implies that $\Ha^s(A) = \Ha^s_\infty(A)$.
\end{proof}

\begin{exemple}
If \( s < N \) and if $f : \R^N \to \R$ is a nonnegative Borel measurable function, then the absolutely continuous measure $\lambda_f$ defined by
\[
	\lambda_f(E) = \int_E f~\d x \quad \text{for every Borel set \( E \subset \R^N \)}
\]
satisfies \( \lambda_f \leq \Ha^s \).
Indeed, for every Borel set \( E \subset \R^N \) with finite \( \Ha^s \) measure, we have \(\Ha^N(E) = 0 \) and thus \( \lambda_f(E) = 0 \). 
Now, using Proposition~\ref{characsstraightmes} below, one shows that \( \lambda_f \leq \Ha^s_\infty \) if and only if 
\[
	\int_{B_r(x)} f~\d x \leq \omega_s r^s \quad \text{for every ball \(B_r(x) \subset \R^N\).}
\]
This condition is fulfilled for instance if 
\[
	f(x) \leq \frac{s\omega_s}{N\omega_N} \frac{1}{\norm{x}^{N-s}} \quad \text{for almost every \( x \in \R^N \setminus \{0\} \).}
\]
\end{exemple}

In the previous example, we used that, as for $s$-straight sets~\cite{delaware}*{Theorem~1}, Borel measures \( \mu \leq \Ha^s_\infty \) are characterized in terms of a density estimate:

\begin{proposition}
\label{characsstraightmes}
If\/ \( 0 < \delta \leq +\infty \) and if $\mu$ is a Borel measure on $\R^N$, then \( \mu \leq \Ha^s_\delta\) if and only if 
\[
	\mu(B_r(x)) \leq \omega_s r^s \quad \text{for every ball \( B_r(x) \subset \R^N \) with \(0 \leq r \leq \delta\).}
\]
\end{proposition}
\begin{proof}
If \( \mu \leq \Ha^s_\delta \), then we may directly write
\[
	\mu(B_r(x)) \leq \Ha^s_\delta(B_r(x)) \leq \omega_s r^s 
\]
for every ball \( B_r(x) \subset \R^N \) with \(0 \leq r \leq \delta\) by covering $B_r(x)$ with itself in the definition of the Hausdorff capacity.
For the converse, let $E \subset \R^N$ be a Borel set and let $E \subset \bigcup\limits_{n \in \N} B_{r_n}(x_n)$ be an arbitrary covering of $E$ by balls with \( 0 \leq r_n \leq \delta \) for each \(n \in \N\).
The estimate for $\mu$ implies that
\[
	\mu(E) \leq \sum_{n \in \N}\mu(B_{r_n}(x_n)) \leq \sum_{n \in \N} \omega_s r_n^s.
\]
The conclusion follows by taking the infimum over such coverings. 
\end{proof}

This proposition sheds some light on Theorem~\ref{thmdecomposition}. 
Indeed, it allows one, through a decomposition, to convert the abstract inequality $\mu \leq \Ha^s$ into a more explicit estimate satisfied by countably many parts of \( \mu \). 

We conclude this section by an example showing that the set \(F\) in Theorem~\ref{thmdecomposition2} cannot be avoided for a general Borel measure:

\begin{exemple}
Given $s < N$, define a Borel measure $\mu$ on $\R^N$ by letting
\[
	\mu(E) = 
	\begin{cases}
		0 & \text{if } E \text{ is } \sigma\text{-finite under } \Ha^s, \\
		+\infty & \text{otherwise.}
	\end{cases}
\]
This defines a Borel measure on $\R^N$ that only takes the values $0$ and $+\infty$.
Furthermore, $\mu$ does not identically vanish due to our choice of $s$, and we also have $\mu \leq \Ha^s$.

However, the estimate \( \mu\restr{E} \leq \Ha^s_{\infty} \) fails for every Borel set \( E \subset \R^N \) with \( \mu(E) > 0 \).
Indeed, we must have \( \mu(E) = +\infty \) and then, using the Increasing Set Lemma, there exists a cube \( Q \subset \R^N \) such that 
\[
	\mu\restr{E}(Q) = \mu(E \cap Q) = +\infty > \Ha^s_{\infty}(Q).
\]
Hence, for this measure, Theorem~\ref{thmdecomposition} holds with \(I = \varnothing\) and \(F = \R^N\).
\end{exemple}

\section{Existence of an \(s\)-straight part of a measure}
\label{section3}

For convenience, we shall say that a finite Borel measure \(\mu\) is \emph{\(s\)-straight} whenever \( \mu \leq \Ha^s_\infty \).
As we have seen in Example~\ref{sStraightSetsAndMeasures}, this is a natural generalisation of \(s\)-straight sets.

In this section, we show that any nonzero finite Borel measure \(\mu\) contains an \(s\)-straight part, that lives on a Borel set of positive measure under \(\mu\).
This is the counterpart of~\cite{delaware}*{Theorem~4} and will serve as a key tool to prove our decomposition theorems.
The main result of this section is stated as follows:

\begin{proposition}
\label{thm4}
If $\mu$ is a nonzero finite Borel measure on $\R^N$ satisfying $\mu \leq \Ha^s$, then there exists a Borel set $E \subset \R^N$ such that $\mu(E) > 0$ and $\mu\restr{E}$ is $s$-straight.
\end{proposition}

Before proving Proposition~\ref{thm4}, we will need some preparatory tools.
We begin with the next result, which is an analogue of the Intermediate Value Theorem for Borel measures that do not charge singletons.

\begin{proposition}
\label{intermediatevaluemes}
Let $\mu$ be a finite Borel measure on $\R^N$ such that, for every $x \in \R^N$, we have $\mu(\{x\}) = 0$.
If $E \subset \R^N$ is a Borel set satisfying $\mu(E) > 0$, then, for every $0 < c < \mu(E)$, there exists a Borel set $A \subset E$ such that $\mu(A) = c$.
\end{proposition}

This property is well-known, but we present a proof for the comfort of the reader.
It relies on the following lemma.

\begin{lemme}
\label{lemmaintermediatevaluemes}
Let $\mu$ be a finite Borel measure on $\R^N$ such that, for every $x \in \R^N$, we have $\mu(\{x\}) = 0$.
If $E \subset \R^N$ is a Borel set satisfying $\mu(E) > 0$, then, for each $\varepsilon > 0$, there exists a Borel set $A \subset E$ such that $0 < \mu(A) \leq \varepsilon$.
\end{lemme}

\begin{proof}[Proof of Lemma~\ref{lemmaintermediatevaluemes}]
Using the Increasing Set Lemma, one finds a cube $Q_0 \subset \R^N$ such that $\mu(E \cap Q_0) > 0$.
Then, by induction, one constructs a decreasing sequence of cubes $(Q_n)_{n \in \N}$ such that $\mu(E \cap Q_n) > 0$ for each $n \in \N$ and where the size of the side lengths of the cubes are halved at each step.

Since the diameters of the cubes tend to $0$ as $n \to \infty$, there exists $x \in \R^N$ such that $\bigcap\limits_{n \in \N} Q_n \subset \{x\}$.
As $\mu$ is finite, one may apply the Decreasing Set Lemma to obtain
\[
	\lim_{n \to \infty} \mu(E \cap Q_n) = \mu\bigg(E \cap \bigg(\bigcap_{n \in \N} Q_n\bigg)\bigg) \leq \mu(\{x\}) = 0.
\]
It thus suffices to take $A = E \cap Q_m$ for $m \in \N$ large enough.
\end{proof}

The proof of Proposition~\ref{intermediatevaluemes} is based on a standard exhaustion argument.

\begin{proof}[Proof of Proposition~\ref{intermediatevaluemes}]
Let $0 < c < \mu(E)$ and $A_0 = \varnothing$.
By induction, we may construct a sequence $(A_n)_{n \in \N}$ of disjoint Borel subsets of $E$ such that 
\[
	\frac{c_n}{2} \leq \mu(A_{n}) \leq c_n \leq c \quad \text{for each $n \in \N_\ast$},
\]
where
\[
	c_n = \sup\imset{\mu(A)}{A \subset E \setminus \bigcup_{k = 0}^{n-1} A_k,~\mu(A) \leq c - \mu\bigg(\bigcup_{k = 0}^{n-1} A_k\bigg)}.
\]
Let $A = \bigcup\limits_{n \in \N} A_n \subset E$, so that $\mu(A) = \lim\limits_{n \to \infty} \mu\Big(\bigcup\limits_{k = 0}^{n} A_k\Big) \leq c$.
Assume by contradiction that this inequality is strict.
Then, since 
\[
	\mu(E \setminus A) = \mu(E) - \mu(A) > 0 \quad \text{ and } \quad c - \mu(A) > 0,
\]
we may invoke the previous lemma to obtain a Borel set $B \subset E \setminus A$ such that $0 < \mu(B) \leq c - \mu(A)$.
But this implies that $B$ is admissible in the definition of the numbers $c_n$, whence we deduce that $\mu(B) \leq c_n$ for each $n \in \N_\ast$.
On the other hand, we write
\[
	\frac{1}{2}\sum_{n \in \N_\ast} c_n \leq \sum_{n \in \N_\ast} \mu(A_n) = \mu(A) < +\infty,
\]
and thus $c_n \to 0$ as $n \to \infty$.
This contradicts the fact that $0 < \mu(B) \leq c_n$ for each $n \in \N_\ast$, and the conclusion follows.
\end{proof}

Another tool is the following result whose proof can be found in~\cite{elliptic_pdes}*{Proposition~14.15}.

\begin{proposition}
\label{prop1415}
If $\mu$ is a finite Borel measure on $\R^N$ satisfying $\mu \leq \Ha^s$, then, for every $\varepsilon > 0$, there exists a Borel set $A \subset \R^N$ such that
\begin{enumerate}
	\item for every $\beta > 1$, there exists $\delta > 0$ such that $\mu\restr{A} \leq \beta\Ha^s_\delta$\ ,
	\item $\mu(\R^N \setminus A) \leq \varepsilon$.
\end{enumerate}
\end{proposition}

We are now ready to prove Proposition~\ref{thm4}.
We rely on the strategy used by R. Delaware that had been suggested to him by D. Preiss.
In our case, Proposition~\ref{prop1415} above acts as a replacement for general measures of~\cite{delaware}*{Lemma~1} for \( \Ha^s \).

\begin{proof}[Proof of Proposition~\ref{thm4}]
We begin by dealing separately with the special case where there exists $x \in \R^N$ such that $\mu(\{x\}) > 0$.
This can only happen when $s = 0$, but then we have
\[
	0 < \mu(\{x\}) \leq \Ha^s(\{x\}) = 1 = \Ha^s_\infty(\{x\}).
\]
It therefore suffices to let $E = \{x\}$ to conclude.
We may thus assume from now on that $\mu(\{x\}) = 0$ for every $x \in \R^N$, which implies that $\mu$ satisfies the assumptions of Proposition~\ref{intermediatevaluemes}.

Let $(\varepsilon_j)_{j \in \N}$ be a decreasing sequence of positive real numbers that converges to \(0\). 
We apply Proposition~\ref{prop1415} with any $0 < \varepsilon < \mu(\R^N)$ to extract a Borel set $A \subset \R^N$ such that $\mu(A) > 0$ and a sequence $(\delta_j)_{j \in \N}$ such that $\mu\restr{A} \leq (1 + \varepsilon_j)\Ha^s_{\delta_j}$ for each $j \in \N$.
By Proposition~\ref{characsstraightmes}, this condition rewrites as
\begin{equation}
\label{estimatemurestrA}
	\mu\restr{A}(B_r(x)) \leq (1 + \varepsilon_j)\omega_s r^s
\end{equation}
for every ball $B_r(x) \subset \R^N$ with $r \leq \delta_j$.

We now turn to the construction of the required Borel set $E$.
Let us first pick a decreasing sequence of positive real numbers $(r_j)_{j \in \N}$ with $r_j \leq \delta_j$ for each $j \in \N$ that satisfies $\lim\limits_{j \to \infty} r_j = 0$ in addition to other properties to be described later on.
For each $j \in \N_\ast$, we consider a partition $A = \bigcup\limits_{i \in \N} A_{i,j}$ into disjoint Borel sets with diameter less than $r_{j+1}$.
Assume that the sequence $(\varepsilon_j)_{j \in \N}$ has been chosen such that $\varepsilon_j < 1$ for every $j \in \N_\ast$.
We may thus apply Proposition~\ref{intermediatevaluemes} to find, for each $i \in \N$ and $j \in \N_\ast$, a Borel set $E_{i,j} \subset A_{i,j}$ such that
\[
	\mu(E_{i,j}) = (1 - \varepsilon_{j-1})\mu(A_{i,j}).
\]
We then let
\[
	E = \bigcap_{j \in \N_\ast}\bigcup_{i \in \N} E_{i,j}.
\]

We show that the Borel set $E$ satisfies the conclusion of Proposition~\ref{thm4}, provided that the sequences $(\varepsilon_j)_{j \in \N}$ and $(r_j)_{j \in \N}$ are suitably chosen.
We begin with $\mu(E) > 0$.
Since for each $j \in \N_\ast$ the sets $A_{i,j}$ are disjoint, the same holds for the subsets $E_{i,j}$, whence
\[
	\mu\bigg(\bigcup_{i \in \N} E_{i,j}\bigg) = \sum_{i \in \N}\mu(E_{i,j}) = \sum_{i \in \N}(1 - \varepsilon_{j-1})\mu(A_{i,j}) = (1 - \varepsilon_{j-1})\mu(A)
\]
for each $j \in \N_\ast$.
Writing
\[
	A \setminus E = A \setminus \bigg(\bigcap_{j \in \N_\ast}\bigcup_{i \in \N} E_{i,j}\bigg) = \bigcup_{j \in \N_\ast}\bigg(A \setminus \bigcup_{i \in \N} E_{i,j}\bigg),
\]
we obtain
\[
\begin{split}
	\mu(A \setminus E) &\leq \sum_{j \in \N_\ast} \mu\bigg(A \setminus \bigcup_{i \in \N} E_{i,j}\bigg) \\ 
	&= \sum_{j \in \N_\ast}(\mu(A) - (1 - \varepsilon_{j-1})\mu(A)) = \sum_{j \in \N_\ast}\varepsilon_{j-1}\mu(A).
\end{split}
\]
So, if we choose the sequence $(\varepsilon_j)_{j \in \N}$ such that 
\[
	\sum_{j \in \N} \varepsilon_j < 1,
\]
then we obtain 
\[
	\mu(E) = \mu(A) - \mu(A \setminus E) > 0.
\]

It remains to prove that $\mu\restr{E}$ is $s$-straight.
According to Proposition~\ref{characsstraightmes}, it suffices to show that $\mu\restr{E}(B_r(x)) \leq \omega_s r^s$ for every ball $B_r(x) \subset \R^N$ with \( r > 0 \).
Assume that $r \geq r_0$.
Using Proposition~\ref{intermediatevaluemes}, we may suppose from the beginning that $0 < \mu(A) \leq \omega_s r_0^s$.
We therefore have
\[
	\mu\restr{E}(B_r(x)) \leq \mu(E) \leq \mu(A) \leq \omega_s r_0^s \leq \omega_s r^s.
\]

We are thus left with the case $0 < r < r_0$.
Since the sequence $(r_j)_{j \in \N}$ decreases to $0$, we may then find $j \in \N_\ast$ such that 
\[
	r_j \leq r < r_{j-1}.
\]
Denote by $I$ the set of indices $i \in \N$ such that $B_r(x) \cap E_{i,j} \neq \varnothing$.
We observe that
\begin{align*}
\mu\restr{E}(B_r(x)) &= \mu\bigg(B_r(x) \cap \bigcap_{k \in \N_\ast}\bigcup_{i \in \N} E_{i,k}\bigg) \\ 
&\leq \mu\bigg(\bigcup_{i \in \N} (B_r(x) \cap E_{i,j})\bigg) = \mu\bigg(\bigcup_{i \in I} (B_r(x) \cap E_{i,j})\bigg).
\end{align*}
By the subadditivity of $\mu$, we get
\[
	\mu\restr{E}(B_r(x)) \leq \sum_{i \in I}\mu(B_r(x) \cap E_{i,j}) \leq \sum_{i \in I} \mu(E_{i,j}) = (1 - \varepsilon_{j-1})\sum_{i \in I} \mu(A_{i,j}).
\]
Since the sets $A_{i,j}$ are disjoint, we find
\[	
	\mu\restr{E}(B_r(x)) \leq (1 - \varepsilon_{j-1})\mu\bigg(\bigcup_{i \in I} A_{i,j}\bigg).
\]
Each set $A_{i,j}$ having a diameter less than $r_{j+1}$, it follows from the definition of $I$ that $A_{i,j} \subset B_{r + r_{j+1}}(x) \cap A$ for every $i \in I$, and thus $\bigcup\limits_{i \in I} A_{i,j} \subset B_{r + r_{j+1}}(x) \cap A$.
We deduce that
\[
	\mu\restr{E}(B_r(x)) \leq (1 - \varepsilon_{j-1})\mu(B_{r + r_{j+1}}(x) \cap A) = (1 - \varepsilon_{j-1})\mu\restr{A}(B_{r + r_{j+1}}(x)).
\]
Assume now that
\[
	2r_k \leq \delta_k \quad \text{for each \( k \in \N \).}
\]
Since $r \leq r_{j-1}$, we have $r + r_{j+1} \leq 2r_{j-1} \leq \delta_{j-1}$.
We may thus apply estimate (\ref{estimatemurestrA}) with \( j - 1 \) to find
\[
	\mu\restr{E}(B_r(x)) \leq (1 - \varepsilon_{j-1})(1 + \varepsilon_{j-1})\omega_s(r + r_{j+1})^s = (1 - \varepsilon_{j-1}^2)\omega_s(r + r_{j+1})^s.
\]
To conclude, we invoke the uniform continuity of the function $t \mapsto t^s$ over $[r_k, r_{k-1}]$ to construct inductively \(r_{k+1} < r_k \) so that
\[
	\frac{(t + r_{k+1})^s}{t^s} \leq \frac{1}{1 - \varepsilon_{k - 1}^2} \quad \text{for every $r_k \leq t \leq r_{k-1}$.}
\]
This leads to \( \mu\restr{E}(B_r(x)) \leq \omega_s r^s \) as required.
Thus, \( \mu\restr{E} \) is \(s\)-straight.
\end{proof}

\section{Proofs of Theorems~\ref{thmdecomposition} and~\ref{thmdecomposition2}}
\label{section4}

This section is devoted to the proofs of our main results.
Theorem~\ref{thmdecomposition} is proved by an exhaustion argument, using Proposition~\ref{thm4} to recursively extract \(s\)-straight parts of the Borel measure \(\mu\). 

\begin{proof}[Proof of Theorem~\ref{thmdecomposition}]
Let $E_0 = \varnothing$.
By induction, we construct a sequence of disjoint Borel sets $(E_n)_{n \in \N}$ such that $\mu\restr{E_n}$ is $s$-straight for each $n \in \N$ and $\frac{1}{2}d_n \leq \mu(E_n) \leq d_n$ for each $n \in \N_\ast$, where
\[
	d_n = \sup\imset{\mu(A)}{A \subset \R^N \setminus \bigcup_{k = 0}^{n-1} E_k,~\mu\restr{A} \text{ is }s\text{-straight}}.
\]

We claim that $\mu\Big(\R^N \setminus \bigcup\limits_{n \in \N} E_n\Big) = 0$.
Assume by contradiction that this statement fails.
Using Proposition~\ref{thm4}, we may thus find a Borel set $A \subset \R^N \setminus \bigcup\limits_{n \in \N} E_n$ such that $\mu(A) > 0$ and $\mu\restr{A}$ is $s$-straight.
But then, $A$ is admissible in the definition of the numbers $d_n$, so that $\mu(A) \leq d_n$ for each $n \in \N_\ast$.
On the other hand, since
\[
	\frac{1}{2}\sum_{n \in \N_\ast} d_n \leq \sum_{n \in \N_\ast} \mu(E_n) = \mu\bigg(\bigcup_{n \in \N_\ast} E_n\bigg) < +\infty,
\]
we deduce that the sequence $(d_n)_{n \in \N_\ast}$ tends to $0$, which contradicts the fact that $0 < \mu(A) \leq d_n$ for each $n \in \N_\ast$.

Since $\mu\Big(\R^N \setminus \bigcup\limits_{n \in \N} E_n\Big) = 0$, the measure $\mu\restr{\R^N \setminus \bigcup\limits_{n \in \N} E_n}$ is identically zero and in particular \(s\)-straight.
We may therefore replace \( E_0 = \varnothing \) by $\R^N \setminus \bigcup\limits_{n \in \N} E_n$ in the sequence $(E_n)_{n \in \N}$, and this provides the required sequence of Borel sets.
\end{proof}

Theorem~\ref{thmdecomposition} can be readily extended to the \(\sigma\)-finite case.
Indeed, if \(\mu\) is \(\sigma\)-finite, one finds a sequence of disjoint Borel sets \( (A_n)_{n \in \N} \) of finite measure under \(\mu\).
One may then apply Theorem~\ref{thmdecomposition} to \(\mu\restr{A_n}\) for each \( n \in \N \) to obtain a sequence of disjoint Borel sets \( (E_{k,n})_{k \in \N} \) such that \(\R^N = \bigcup\limits_{k \in \N} E_{k,n} \) and \( \big(\mu\restr{A_n}\big)\restr{E_{k,n}} = \mu\restr{A_n \cap E_{k,n}} \) is \(s\)-straight for each \(k \in \N\).
It now suffices to reorder the countable family \( \{A_n \cap E_{k,n}\}_{(k,n) \in \N^2} \) into a sequence to get the conclusion.

We now show how Theorem~\ref{thmdelaware} can be obtained as a corollary of the previous theorem.

\begin{proof}[Proof of Theorem~\ref{thmdelaware}]
If $E \subset \R^N$ is a Borel set of finite $\Ha^s$ measure, then $\Ha^s\restr{E}$ is a finite Borel measure on $\R^N$ satisfying $\Ha^s\restr{E} \leq \Ha^s$.
Hence, Theorem~\ref{thmdecomposition} ensures the existence of a sequence of disjoint Borel sets $(E_n)_{n \in \N}$ such that $\R^N = \bigcup\limits_{n \in \N} E_n$ and $\big(\Ha^s\restr{E}\big)\restr{E_n} = \Ha^s\restr{E_n \cap E}$ is $s$-straight for each $n \in \N$.
But this implies that $E_n \cap E$ is $s$-straight for each $n \in \N$.
\end{proof}

We turn to the proof of our decomposition theorem for general Borel measures. 
The argument follows the same idea as for Theorem~\ref{thmdecomposition}, but one has to replace the by-hand inductive construction with Zorn's lemma.

\begin{proof}[Proof of Theorem~\ref{thmdecomposition2}]
Let $\mathcal{F}$ be the set of all families $(E_i)_{i \in I}$ of disjoint Borel sets of positive measure under $\mu$ and such that $\mu\restr{E_i}$ is $s$-straight for each $i \in I$.
Observe that $\mathcal{F}$ contains the empty family.
We introduce a partial order in \( \mathcal{F} \) by defining \( (E_i)_{i \in I} \leq (F_j)_{j \in J} \) whenever \( I \subset J \) and \( E_i = F_i \) for every \( i \in I \).
We notice that $\mathcal{F}$ is inductive since an upper bound for a chain \( \{(E_i)_{i \in I_\alpha} \}_{\alpha \in A} \) in $\mathcal{F}$ is provided by \( (E_i)_{i \in \bigcup\limits_{\alpha \in A} I_\alpha} \).
Hence, Zorn's lemma ensures the existence of a maximal element $(E_i)_{i \in I}$ in $\mathcal{F}$.

Let $F = \R^N \setminus \bigcup\limits_{i \in I} E_i$.
Assume by contradiction that $\mu\restr{F}$ does not take only the values $0$ and $+\infty$. 
In such a case, there exists $E \subset F$ such that $0 < \mu(E) < +\infty$.
Thus, Proposition~\ref{thm4} guarantees the existence of a Borel set $A \subset E$ such that $\mu(A) > 0$ and $\mu\restr{A}$ is $s$-straight.
But then adding $A$ to the family $(E_i)_{i \in I}$ contradicts its maximality in $\mathcal{F}$.
Therefore, $F$ is the required Borel set.
\end{proof}

When \(\mu\) is \(\sigma\)-finite, the index set \(I\) in Theorem~\ref{thmdecomposition2} is countable and the set \(F\) is negligible under \(\mu\).
One then recovers Theorem~\ref{thmdecomposition} for \(\sigma\)-finite measures.
Indeed, since \(\mu\) is \(\sigma\)-finite, we may find a sequence of Borel sets \( (A_n)_{n \in \N} \) of finite measure under \(\mu\) such that \( \R^N = \bigcup\limits_{n \in \N} A_n \).

To see that \(F\) is negligible under \(\mu\), use the subadditivity of \(\mu\) to write \( \mu(F) \leq \sum\limits_{n \in \N} \mu(F \cap A_n) \).
Since \( \mu(F \cap A_n) < +\infty \) for each \( n \in \N \), we have \( \mu(F \cap A_n) = 0 \) by assumption on \(F\).
Thus, \( \mu(F) = 0 \).

To prove that \( I \) is countable, we first observe that \( I \subset \bigcup\limits_{n \in \N} I_n \), where each \( I_n \) denotes the set of indices \( i \in I \) with \( \mu(E_i \cap A_n) > 0 \).
Such an inclusion comes from the fact that \( \mu(E_i) > 0 \) for every \( i \in I \).
To conclude, it then suffices to check that each \( I_n \) is countable, which follows from \( \mu(A_n) < +\infty \).

\medskip

In this paper, we have relied on the original definition of the Hausdorff measure, sometimes called \emph{spherical Hausdorff measure}.
Another common definition uses coverings by arbitrary subsets of \(\R^N\), and is based on the following Hausdorff capacity : 
\[
	\widetilde{\Ha}^s_\delta(A) = \inf\imset{\sum_{n \in \N} (\diam A_n)^s}{E \subset \bigcup_{n \in \N} A_n,~ 0 \leq \diam A_n \leq \delta},
\]
where \( \diam A \) denotes the diameter of the set \( A \subset \R^N \).
The Hausdorff measure \( \widetilde{\Ha}^s \) is then defined accordingly.
The counterparts of Theorems~\ref{thmdecomposition} and~\ref{thmdecomposition2} are true using \( \widetilde{\Ha}^s \) and \( \widetilde{\Ha}^s_\infty \) instead.
The proof of Proposition~\ref{thm4} requires some minor changes.
This comes from the fact that the characterisation of \(s\)-straight measures in Proposition~\ref{characsstraightmes} has to be replaced by the estimate
\[      
        \mu(A) \leq (\diam A)^s \quad \text{for every Borel set \( A \subset \R^N \) with \( 0 \leq \diam A \leq \delta \).}
\]
Equation~\eqref{estimatemurestrA} is thus modified accordingly, and the last part of the proof consists in showing that
\[
	\mu\restr{E}(B) \leq (\diam B)^s \quad \text{for every Borel set \( B \subset \R^N \) with \( \diam B > 0 \).}
\]
For this purpose, in the case \( \diam B \geq r_0 \) we assume that \( \mu(A) \leq r_0^s \) instead of \( \mu(A) \leq \omega_s r_0^s \), while in the case \( 0 < \diam B < r_0 \) we replace \( B_{r + r_{j+1}}(x) \) by \( B' = \compset{x \in \R^N}{\d(x, B) \leq r_{j+1}} \).
The other proofs remain unchanged.


\section{Proof of Theorem~\ref{thmBLOP}}
\label{sectionBLOP}

In this last section, we make use of the decomposition we obtained in Theorem~\ref{thmdecomposition} to prove the existence of a solution in the sense of distributions for the Dirichlet problem \eqref{eqPDE-BLOP}
involving a finite Borel measure \( \nu \) in \( \Omega \subset \R^N \).

To this end, we need an exponential estimate involving the Newtonian potential \(\Newton\nu\).
We recall that \(\Newton\nu : \R^{N} \to [0, +\infty]\) is defined in dimension \(N \ge 3\) for every \(x \in \R^{N}\) by
\[{}
\Newton\nu(x){}
= E * \nu(x){}
= \frac{1}{(N-2)N\omega_{N}} \int_{\Omega} \frac{\d\nu(y)}{|x - y|^{N - 2}},
\]
where \(E\) is the fundamental solution of \(-\Delta\).{}
In dimension \(N = 2\), the Newtonian potential has an analogous definition involving the \(\log\) function.
Assuming that \( \nu \leq \alpha\Ha_{\delta}^{N-2} \) for some \( \alpha < 4\pi \) and \( 0 < \delta \leq +\infty \),
one shows that 
\begin{equation}
\label{eqBLOP-Estimate}
	\e^{\Newton\nu} \in L^{1}_{\mathrm{loc}}(\R^N).
\end{equation}
The proof can be found in \cite{elliptic_pdes}*{Chapter~17} that is modeled upon \cite{BLOP}.
Observe that \eqref{eqBLOP-Estimate} is false for example if \(\nu = 4\pi \Ha^{N-2}\lfloor_{M^{N-2}}\), where \(M^{N-2}\) is a (non-empty) compact manifold of dimension \(N-2\).{}

We begin with a particular case of Theorem~\ref{thmBLOP}:

\begin{lemme}
	\label{lemmaBLOP}
	Theorem~\ref{thmBLOP} holds when \(\nu \le \alpha \Ha_{\delta}^{N-2}\) for some \(\alpha < 4\pi\) and \(0 < \delta \le +\infty\).
\end{lemme}

The proof of Lemma~\ref{lemmaBLOP} relies on various properties about Dirichlet problems with an absorption term that includes \eqref{eqPDE-BLOP} as a particular case.
We refer the reader to the book \cite{elliptic_pdes}, more specifically Chapters~4, 5, 6 and 19, for details.

\begin{proof}[Proof of Lemma~\ref{lemmaBLOP}]
	Take a sequence of radially decreasing mollifiers \((\rho_{k})_{k \in \N}\) in \(\R^{N}\).{}
	For each \(k \in \N\), let \(u_{k}\) satisfy
	\begin{equation}
	\label{mollified-eqPDE-BLOP}
	\left\{
\begin{alignedat}{2}
	- \Delta u_{k} + (\mathrm{e}^{u_{k}} - 1) & = \rho_{k} * \nu &&\quad \text{in \(\Omega\),}\\
	u_{k} & = 0 &&\quad \text{on \(\partial\Omega\).}
\end{alignedat}
\right.
	\end{equation}
	The existence of \(u_{k}\) can be deduced by minimisation of the functional \(\mathcal{E} : W_{0}^{1, 2}(\Omega) \to (-\infty, +\infty]{}\) defined by
	\[{}
	\mathcal{E}(v)
	= \frac{1}{2} \int_{\Omega} |\nabla v|^{2} + \int_{\Omega} (\e^{v} - v)
	- \int_{\Omega} (\rho_{k} * \nu) v.
	\]
	From the Euler-Lagrange equation satisfied by \( u_{k} \), we have
	\begin{equation}
		\label{eq-794}
	- \int_{\Omega} u_{k} \Delta\varphi + \int_{\Omega} (\e^{u_{k}} - 1) \varphi{}
	= \int_{\Omega} (\rho_{k} * \nu) \varphi{}
	\quad \text{for every \(\varphi \in C_{\mathrm{c}}^{\infty}(\Omega)\)},
	\end{equation}
	where \( C_{\mathrm{c}}^{\infty}(\Omega) \) is the set of smooth functions with compact support in \( \Omega \).
	Using the absorption estimate
	\[{}
	\|\mathrm{e}^{u_{k}} - 1\|_{L^{1}(\Omega)}
	\le \|\rho_{k} * \nu\|_{L^{1}(\Omega)} 
	\le \nu(\Omega),
	\]
	we deduce from the equation and the triangle inequality that
	\[{}
	\|\Delta u_{k}\|_{L^{1}(\Omega)}
	\leq \|\rho_{k} * \nu\|_{L^{1}(\Omega)} + \|\mathrm{e}^{u_{k}} - 1\|_{L^{1}(\Omega)}
	\le 2 \nu(\Omega).
	\]
	Littman-Stampacchia-Weinberger's estimate ensures that \((u_{k})_{k \in \N}\) is a bounded sequence in \(W_{0}^{1, p}(\Omega)\) for every \( 1 \leq p < \frac{N}{N-1} \).{}
	By the Rellich-Kondrashov Compactness Theorem~\cite{willem}*{Theorem~6.4.6}, there exist a subsequence \((u_{k_{j}})_{j \in \N}\) and \(u \in W_{0}^{1, 1}(\Omega)\) such that
	\(u_{k_{j}} \to u\) in \(L^{1}(\Omega)\) and almost everywhere in \(\Omega\).{}
	
	Since \(\Newton(\rho_{k} * \nu)\) is a supersolution of the Dirichlet problem \eqref{mollified-eqPDE-BLOP}, by comparison we have
	\[{}
	0 \le u_{k} \le \Newton(\rho_{k} * \nu).
	\]
	Since the fundamental solution \(E\) is superharmonic and \( \rho_{k} \) is radially decreasing, \( E * \rho_{k} \leq E \).
	An application of Fubini's theorem then gives
	\[{}
	\Newton(\rho_{k} * \nu){}
	= E * (\rho_{k} * \nu) 
	= (E * \rho_{k}) * \nu{}
	\le E * \nu{}
	= \Newton\nu.
	\]
	Hence,
	\[{}
	0 \le u_{k} \le \Newton\nu.
	\]
	By the density assumption satisfied by \(\nu\), we can apply estimate \eqref{eqBLOP-Estimate} to deduce that \(\mathrm{e}^{\Newton\nu} \in L^{1}(\Omega)\).{}
	Therefore, the Dominated Convergence Theorem ensures that \(\mathrm{e}^{u_{k_{j}}} \to \mathrm{e}^{u}\) in \(L^{1}(\Omega)\).
	The conclusion then follows as we take \(k = k_{j}\) in \eqref{eq-794} and let \(j \to \infty\). 
\end{proof}

	We now present a localised reformulation of Theorem~\ref{thmdecomposition}:

\begin{lemme}
	\label{lemmaDecompositionLocal}
	If \(\mu \le \Ha^{s}\), then there exist a non-increasing sequence of open sets \((U_{j})_{j \in \N}\) and a sequence of positive numbers \((\delta_{j})_{j \in \N}\) such that \(\mu(U_{j}) \to 0\) and
	\[{}
	\mu\lfloor_{\R^{N} \setminus U_{j}}{} \le \Ha_{\delta_{j}}^{s}
	\quad \text{for every \(j \in \N\).}
	\]
\end{lemme}

\begin{proof}
	Let \(\varepsilon > 0\) and consider a decomposition \(\R^{N} = \bigcup\limits_{n \in \N}{E_{n}}\) given by Theorem~\ref{thmdecomposition}.
	For each \(n \in \N\), by inner regularity of \(\mu\) there exists a compact set \(F_{n} \subset E_{n}\) with \(\mu(E_{n} \setminus F_{n}) \le \varepsilon/2^{n+1}\).{}
	We have \(\mu\Bigl(\R^{N} \setminus \bigcup\limits_{n \in \N}F_{n}\Bigr) \le \varepsilon\).{}
	By construction, for any \( \eta > 0 \) we have
	\[{}
	\mu\lfloor_{F_{n}}{}
	\le \Ha_{\infty}^{s}
	\le \Ha_{\eta}^{s}.
	\]
	Given \(k \in \N_{*}\), choosing \(\eta = \eta_{k} := \frac{1}{2}\min{\{\d(F_{\alpha}, F_{\beta}) : 0 \le \alpha < \beta \le k\}}\) in the inequality above implies that
	\[{}
		\mu\lfloor_{\bigcup\limits_{n = 0}^{k}{F_{n}}}
		\le \Ha_{\eta_{k}}^{s}.
	\]
	Thus, for \(k\) sufficiently large, the open set \(U = \R^{N} \setminus \bigcup\limits_{n = 0}^{k}{F_{n}}\) is such that 
	\[{}
	\mu(U) \le 2\varepsilon{}
	\quad \text{and} \quad{}
	\mu\lfloor_{\R^N \setminus U}{} \le \Ha_{\eta_{k}}^{s}.
	\]
	It now suffices to apply this conclusion to a sequence \(\varepsilon_{j} \to 0\) to find a sequence \((U_{j})_{j \in \N}\) satisfying the conclusion, except for the fact that it need not be non-increasing.
	To achieve such an additional property, we further impose that \(\sum\limits_{j = 0}^{\infty}{\varepsilon_{j}} < +\infty\). Thus, \(\sum\limits_{j = 0}^{\infty}{\mu(U_{j})} < +\infty\).{}
	For each \(n \in \N\), let \(O_{n} = \bigcup\limits_{j = n}^{\infty}{U_{j}}\).{}
	The sequence \((O_{n})_{n \in \N}\) satisfies the required properties.
\end{proof}

\begin{proof}[Proof of Theorem~\ref{thmBLOP}]
	We apply Lemma~\ref{lemmaDecompositionLocal} with \(\mu = \nu/4\pi\) to obtain a non-increasing sequence of open sets \((U_{j})_{j \in \N}\) such that \(\nu(U_{j}) \to 0\) and
	\[{}
	\nu\lfloor_{\Omega \setminus U_{j}}{} \le 4\pi \Ha_{\delta_{j}}^{N-2}
	\quad \text{for every \(j \in \N\).}
	\]
	Let \((\beta_{j})_{j \in \N}\) be an increasing sequence of positive numbers such that \(\beta_{j} \to 1\).{}
	For each \(j \in \N\), Lemma~\ref{lemmaBLOP} provides \(w_{j}\) such that
	\[{}
	\left\{
\begin{alignedat}{2}
	- \Delta w_{j} + (\mathrm{e}^{w_{j}} - 1) & = \beta_{j} \nu\lfloor_{\Omega \setminus U_{j}} &&\quad \text{in \(\Omega\),}\\
	w_{j} & = 0 &&\quad \text{on \(\partial\Omega\).}
\end{alignedat}
\right.
	\]
	By monotonicity of \((\beta_{j}\nu\lfloor_{\Omega \setminus U_{j}})_{j \in \N}\) and comparison of solutions, the sequence \((w_{j})_{j \in \N}\) is nondecreasing.
	Denote its pointwise limit by \(w\).
	Using the absorption estimate, we find
	\begin{equation}{}
	\label{firstInequality}
	\|\mathrm{e}^{w_{j}} - 1\|_{L^{1}(\Omega)}
	\le \beta_{j} \nu(\Omega\setminus U_{j})
	\le \nu(\Omega)
	\end{equation}
	and then, by the triangle inequality,
	\begin{equation}{}
	\label{secondInequality}
	\|\Delta w_{j}\|_{L^{1}(\Omega)}
	\le 2 \nu(\Omega).
	\end{equation}
	From \eqref{secondInequality}, we have by Littman-Stampacchia-Weinberger's estimate that \(w \in W_{0}^{1, 1}(\Omega)\).{}
	From \eqref{firstInequality} and Fatou's lemma, we have \(\mathrm{e}^{w} \in L^{1}(\Omega)\).{}
	As \(j \to \infty\) in the integral identity satisfied by \(w_{j}\), we deduce that \(w\) is a solution of the equation in the sense of distributions in \(\Omega\).
\end{proof}


\end{document}